\documentclass[reqno]{amsart}
\usepackage{amssymb,amscd,verbatim, amsthm,graphics, color,latexsym,amsmath,multicol}
\usepackage{fancybox}
\usepackage{longtable}

\usepackage[all]{xy}

\newcommand \fk[1]{{{\mathfrak #1}}}
\newcommand \C[1]{{\mathcal #1}}

\newcommand\fg{\mathfrak g}

\newcommand \bC{{\mathbb C}}

\newcommand \bR{{\mathbb R}}
\newcommand \bZ{{\mathbb Z}}

\newcommand \bQ{{\mathbb Q}}
\newcommand \bN{{\mathbb N}}
\newcommand\bW{{\mathbb W}}

\newcommand\sg{{\mathsf{sig}}}

\newcommand\CB{{\C B}}
\newcommand\CH{{\C H}}

\newcommand\CO{{\C O}}

\newcommand\CX{{\C X}}

\newcommand\al{{\alpha}}

\newcommand\fa{{\mathfrak a}}

\newtheorem{theorem}{Theorem}[section]

\newtheorem{corollary}[theorem]{Corollary}
\newtheorem{lemma}[theorem]{Lemma}
\newtheorem{proposition}[theorem]{Proposition}
\newtheorem{definition}[theorem]{Definition}
\newtheorem{remark}[theorem]{Remark}

\newcommand\Hom{\operatorname{Hom}}
\newcommand\End{\operatorname{End}}

\newcommand\tr{\operatorname{tr}}

\newcommand\Irr{\mathsf{Irr}}
\newcommand\good{\mathsf{good}}

\newcommand\el{\mathsf{ell}}

\newcommand\ind{\mathsf{ind}}

\newcommand\rig{\mathsf{rigid}}
\newcommand\diff{\mathsf{diff}}

\newcommand\Id{\operatorname{Id}}

\newcommand\cind{\operatorname{c-ind}}

\def\<{\langle} 
\def\>{\rangle}

\numberwithin{equation}{section}

\begin{document}

\title{Types and unitary representations of reductive $p$-adic groups}

\author
{Dan Ciubotaru}
        \address[D. Ciubotaru]{Mathematical Institute, University of Oxford, Oxford OX2 6GG, UK}
        \email{dan.ciubotaru@maths.ox.ac.uk}

\begin{abstract}We prove that for every Bushnell-Kutzko type that satisfies a certain rigidity assumption, the equivalence of categories between the corresponding Bernstein component and the category of modules for the Hecke algebra of the type induces a bijection between irreducible unitary representations in the two categories. This is a generalization of the unitarity criterion of Barbasch and Moy for representations with Iwahori fixed vectors.
\end{abstract}

\thanks{}

\subjclass[2010]{22E50}

\maketitle


\section{Introduction}

\subsection{}Let $F$ denote a nonarchimedean local field with finite residue field. Let $G$ be the $F$-points of a connected reductive algebraic group $\C G$ defined over $F$. Fix a Haar measure $\mu$ on $G$. Let $\CH=\CH(G)$ be the Hecke algebra of $G$, i.e., $
\CH(G)$ is the space of locally constant, compactly supported functions $f:G\to \bC$ endowed with the convolution product with respect with $\mu$. Then $\CH$ is an associative nonunital algebra. As it is well known $\CH$ acts on every complex $G$-representation and this induces an equivalence  between the category $\C C(G)$ of smooth (complex) $G$-representations and the category of nondegenerate $\CH$-modules. 

 Let $e\in \CH$ be an idempotent. Examples of idempotents include $e_K=\mu(K)^{-1} \delta_K$, where $\delta_K$ is the indicator function for a compact open subgroup $K$ in $G$. Then $e\CH e$ is a associative subalgebra of $\CH$ and $e\CH e$ has $e$ as its identity element. Let $\Irr G$ denote the set of isomorphism classes of irreducible smooth complex $G$-representations and let $\Irr_e G$ be the subset of irreducible representations $(\pi, V)$ such that $eV=\pi(e)V\neq 0$. Let $\Irr(e\CH e)$ denote the set of (isomorphism classes of) simple $e\CH e$-modules. We have a natural bijection
\begin{equation}\label{e:e-bijection}
\Irr_e G\to \Irr(e\CH e),\quad V\mapsto V_e.
\end{equation}
For every $f\in \CH$, define $f^*\in \CH$ by $f^*(g)=\overline{f(g^{-1})}$, $g\in G$. The operation $*$ is a conjugate-linear anti-involution of $\CH$. Suppose that $e$ is self-adjoint, i.e., $e^*=e$. The algebra $e\CH e$ inherits the operation $*$ and has a natural structure of  a normalized Hilbert algebra. The results of Bushnell, Henniart, and Kutzko \cite{BHK} give an identification between the supports of the Plancherel measures under the bijection (\ref{e:e-bijection}). More precisely, let $\widehat G$ denote the unitary dual of $G$ (the topological space of irreducible unitary representations of $G$ on Hilbert spaces) with the Plancherel measure $\hat \mu$ dual to $\mu$. Let $\widehat G_r$ denote the support of $\widehat \mu$, the space of tempered irreducible $G$-representations. Denote by $\widehat G_r(e)$ the set of representations $(\pi,V)\in \widehat G_r$ such that $\pi(e)\neq 0$. On the other hand, $e\CH e$ has a $C^*$-algebra completion $C^*_r(e\CH e)$ whose dual $\widehat{C^*}_r(e\CH e)$ carries a Plancherel measure $\hat\mu_{e\CH e}$. By \cite[Theorem A]{BHK},  the bijection (\ref{e:e-bijection}) induces a homeomorphism
\begin{equation}
\widehat m_e: \widehat G_r(e)\to C^*_r(e\CH e),
\end{equation}
such that for every Borel set $S$ of $\widehat G_r(e)$, $\widehat\mu(S)=e(1) \widehat\mu_{e\CH e}(\widehat m_e(S))$. In other words, (\ref{e:e-bijection}) induces a natural bijection between irreducible tempered representations.

It is natural to ask if in fact (\ref{e:e-bijection}) induces a bijection of irreducible unitary representations. (We are identifying here preunitary smooth $G$-representations with unitary $G$-representations.) In complete generality, this is clearly false, as seen, for example, by taking $e=e_{K_0}$ where $K_0$ is a maximal special compact open subgroup. In that case, if $V$ is any irreducible representation with a $*$-invariant hermitian form such that $V^{K_0}\neq 0$, then $V^{K_0}$ is automatically a one-dimensional unitary $e_{K_0}\CH e_{K_0}$-module.

\subsection{}For an idempotent $e$, define $\C C_e(G)$ to be the full subcategory of $\C C(G)$ consisting of representations $(\pi,V)$ such that $eV=V$. Following \cite{BHK}, we say that $e$ is {\it special} if $\C C_e(G)$ is closed relative to the formation of $G$-subquotients. This is equivalent \cite[3.12]{BK} to the requirement that the functor
\begin{equation}
m_e:\C C_e(G)\to e\CH e-\text{Mod},\quad V\mapsto eV
\end{equation}
is an equivalence of categories. This situation applies to the idempotents coming from the theory of types initiated by Howe and Moy \cite{HM} and Bushnell and Kutzko, see for example \cite{BK}. This is a generalization of the Borel and Casselman equivalence of categories for representations with Iwahori fixed vectors \cite{Bo}. Let $\C K$ be a compact open subgroup and $(\rho,W)$ a smooth irreducible (finite dimensional) $\C K$-representation. Let $(\rho^\vee,W^\vee)$ be the contragredient representation. The pair $(\C K,\rho)$ is called a {\it type} if \begin{equation}\label{e:K-idempotent}
e_\rho(x)=\mu(\C K)^{-1} \dim W \tr_W(\rho(x^{-1}))\delta_{\C K}(x),\ x\in G,\
\end{equation}
 is a special idempotent. 
 Our main result is the following:

\begin{theorem}\label{t:main}
Suppose that $e\in \CH$ is a self-adjoint special idempotent such that $e=e_\rho$ for a type $(\C K,\rho)$. If $(\C K,\rho) $ is \emph{rigid} (in the sense of Definition \ref{d:rigid-type}), then $m_e$ induces a bijection between unitary representations in $\C C_e(G)$ and unitary $e\CH e$-modules.
\end{theorem}
In section \ref{s:rigid-types}, we show that, in practice, the notion of rigid type is not restrictive. When $G=GL(n,F)$, it is automatic that every type is rigid. More generally, it is easy to see that if $\C K$ is a subgroup of an Iwahori subgroup, then $(\C K,\rho)$ is rigid; this situation covers all types of positive depth in the sense of Moy and Prasad \cite{MP}. On the other hand, if $(\C K,\rho)$ is a level zero type in the sense of Morris \cite{Mo}, then again we can show that $(\C K,\rho)$ is rigid.

\subsection{} Define the $\rho$-spherical Hecke algebra $\CH(G,\rho)$ to be the convolution algebra of locally constant compactly supported functions $f:G\to \End_\bC(W^\vee)$ such that $f(k_1 x k_2)= \rho^\vee(k_1) f(x) \rho^\vee(k_2)$, $k_1,k_2\in \C K$, $x\in G$. This algebra is a normalized Hilbert algebra and it is Morita equivalent with $e_\rho\CH e_\rho$. Set $V_\rho=\Hom_{\C K}[\rho,V]$; this is an $\CH(G,\rho)$-module. As a consequence, we obtain the following equivalence.

\begin{corollary}\label{c:main}
Suppose $(\C K,\rho)$ is a rigid type. Then the equivalence of categories 
\[m_\rho: \C C_{e_\rho}(G)\to \CH(G,\rho)-\text{Mod},\quad V\mapsto V_\rho, 
\]
induces a bijection between unitary representations in the two categories.
\end{corollary}

\subsection{}These results represent a generalization of the Barbasch and Moy preservation of unitarity \cite{BM,BM2} which is the case $e=e_I$ where $I$ is an Iwahori subgroup for a split group $G$. The ideas of \cite{BM,BM2} were extended further in \cite{BC} to situations where $\CH(G,\rho)$ is known to be isomorphic as a $*$-algebra to an affine Hecke algebra with geometric parameters in the sense of Lusztig \cite{Lu,Lu2}. The \cite{BM,BM2} proof in the Iwahori case (and in the generalization \cite{BC}) is based on three main ingredients:
\begin{enumerate}
\item[(a)] Vogan's signature character \cite{Vo};
\item[(b)] the fact that $m_e$ maps irreducible tempered representations to irreducible tempered representations;
\item[(c)] a certain linear independence result proved using Kazhdan-Lusztig theory \cite{KL} and a technical reduction to Lusztig graded affine Hecke algebra \cite{Lu,BM2}.
\end{enumerate}
In particular, this approach is dependent on the knowledge that $e_I\CH e_I$ is a specialization of the generic affine Hecke algebra (possibly with parameters). 

\smallskip

In the present paper, the proof of Theorem \ref{t:main} still relies on Vogan's signature character, but instead of considering $K$-signature characters with respect to the maximal special compact open subgroup $K=K_0$, we consider the signature characters with respect to all conjugacy classes of (maximal) compact open subgroups. The second and essential difference is that the necessary linear independence is obtained as a consequence of the trace Paley-Wiener Theorem proved in \cite{BDK}, see also the work of  Henniart and Lemaire \cite{HL} and \cite{CH2} for more recent accounts and generalizations. We also need to make use of the interplay between the {\it rigid} cocenter and the rigid representation space, in the sense of \cite{CH,CH2}. We emphasize that for this argument, we do not need to know the precise structure of the algebra $e_\rho\CH e_\rho$ or the statement (b) above, nor do we need the reduction to real infinitesimal character for unitary representations of affine Hecke algebras from \cite{BM2}.

\subsection{}As mentioned already, the hypotheses of Theorem \ref{t:main} are known to hold in many situations. For example, one may consider the Moy-Prasad \cite{MP} groups $G_{x,r}$, where $x$ is a point in the Bruhat-Tits building $X(G)$ of $G$ and $r>0$.  Denote $\CH(G,G_{x,r})=e_{G_{x,r}}\CH e_{G_{x,r}}$, which is the Hecke algebra of $G_{x,r}$-biinvariant functions. By \cite[Proposition 5.3]{BS}, the idempotent $e_{G_{x,r}}$ is special. Every $G_{x,r}$, $r>0$ is contained in an Iwahori subgroup. Hence, we have the following:

\begin{corollary}Let $(\pi,V)$ be a irreducible smooth $G$-representation such that $V^{G_{x,r}}\neq 0$, for $x\in X(G)$ and $r>0$. Then $V$ is unitary if and only if $V^{G_{x,r}}$ is a unitary $\CH(G,G_{x,r})$-module.
\end{corollary}

We give a brief outline of the paper. In section \ref{s:signature}, we present the definitions and main properties of the $K$-signature character of a hermitian $G$-representation and adapt Vogan's signature theorem to our setting. The main result of the section is Theorem \ref{t:signature}. In section \ref{s:rigid}, we explain the ideas around the cocenter and trace Paley-Wiener theorem that we need for our applications. The main points are the rigid trace Paley-Wiener theorems for Bernstein components, Corollaries \ref{c:rigid} and \ref{c:rigid-type}. In section \ref{s:proofs}, we use these results to give proofs of Theorem \ref{t:main} and of Corollary \ref{c:main}. Finally, in section \ref{s:rigid-types}, we discuss the notion of rigid type.

\medskip

{\bf Acknowledgements.} I am grateful to Xuhua He for many discussions about the cocenter and the trace Paley-Wiener Theorem. In fact, some of the ideas in this paper originate with the case of the Iwahori component  sketched in \cite[Appendix]{CH}. I thank Guy Henniart for pointing out a gap in an argument in a previous version of the paper; this lead to the introduction of the notion of rigid type. The results of this paper were presented during the special program on representation theory of reductive groups at the Weizmann Institute, Israel, in June 2017.  I thank the organizers for the invitation to attend this excellent workshop and for the support provided by the Weizmann Institute during my stay. 

\smallskip

This research was supported in part by the EPSRC grant EP/N033922/1(2016).

\section{Vogan's signature character}\label{s:signature}
In this section, we recall the definition of the multiplicity $K$-character and signature $K$-character for an admissible $G$-representation $(\pi,V)$ , where $K$ is a compact open subgroup, and explain an analogue of a formula of Vogan for the signature character \cite{Vo,ALTV}.

\subsection{}We begin by recalling  the idea of the Grothendieck group of modules with nondegenerate hermitian forms from \cite{ALTV}. As in the introduction, $\CH$ denotes the Hecke algebra of a reductive $p$-adic group $G$. A star operation $\kappa:\CH\to\CH$ is a conjugate linear involutive algebra anti-automorphism. The main example for us in this paper is \[\kappa(f)=f^*, \quad f^*(g)=\overline{f(g^{-1})}.\]

Let $(\pi,V)$ be a smooth $\CH$-module. A $\kappa$-invariant hermitian form $\langle~,~\rangle_V$ is a hermitian form on $V$ such that
\[\langle\pi(f) v,w\rangle_V=\langle v,\pi(\kappa(f))w\rangle_V, \text{ for all }v,w\in V.
\]
We say that $V$ is hermitian if it has a nondegenerate $\kappa$-invariant hermitian form.
 
Suppose $K$ is a compact open subgroup of $G$. Every smooth irreducible $K$-representation $(\mu,E_\mu)$ is finite dimensional, hence unitary, and we fix a positive definite $K$-invariant form $\langle~,~\rangle_\mu$ on $E_\mu$. Let $\widehat K$ denote the set of equivalence classes of such $K$-representations. If $V$ is admissible, denote by $V(\mu)=\pi(e_\mu)V$ the isotypic component of $\mu$ in $V$, and by
\[V_\mu=\Hom_{\CH(K)}[E_\mu,V]
\]
the multiplicity space of $\mu$. Set $m_V^K(\mu)=\dim V_\mu<\infty$. Then 
\[V=\sum_{\mu\in\widehat K} V(\mu),\quad V(\mu)=V_\mu\otimes E_\mu. 
\]
If $V$ has an invariant hermitian form, let $\langle~,~\rangle_{V(\mu)}$ denote the restriction of the form $\langle~,~\rangle_V$ to $V(\mu)$. This induces a  hermitian form $\langle~,~\rangle^\mu_V$ on $V_\mu$ such that
\[(V(\mu),\langle~,~\rangle_{V(\mu)})=(V_\mu,\langle~,~\rangle^\mu_V)\otimes (E_\mu,\langle~,~\rangle_\mu).
\]
Define $p_V^K(\mu)$, $q_V^K(\mu)$, and $r_V^K(\mu)$ to be the dimension of the positive definite subspace of $(V_\mu,\langle~,~\rangle^\mu_V)$, the dimension of the negative definite subspace, and the dimension of the radical, respectively:
\[m^K_V(\mu)=p_V^K(\mu)+q_V^K(\mu)+r_V^K(\mu).
\]
If $V$ is admissible, let $V^{h,\kappa}$ denote the hermitian dual module. This is an admissible module and $V^{h,\kappa}_{\kappa(\mu)}\cong (V_\mu)^h,$ where $(V_\mu)^h$ is the hermitian dual of the finite dimensional vector space $V_\mu$. Notice that when $V$ is irreducible, $V$ has a nondegenerate $\kappa$-invariant hermitian form if and only if $V\cong V^{h,\kappa}$ as $\CH$-modules (see for example \cite[Proposition 8.9]{ALTV}).

If $U$ is an $\CH$-submodule of $(V,\langle~,~\rangle_V)$, denote by $U^\perp$ the orthogonal complement of $U$ in $V$ with respect to $\langle~,~\rangle_V$. 

If $R$ is any admissible module, define the hyperbolic form on $R\oplus R^{h,\kappa}$, i.e., the nondegenerate $\kappa$-hermitian form $\langle~,\rangle_{hyp}$:
\begin{equation}
\langle R,R\rangle_{hyp}=0,\ \langle R^{h,\kappa},R^{h,\kappa}\rangle_{hyp}=0,\ \langle a,\xi\rangle_{hyp}=\overline{\xi(a)},\ a\in R,\ \xi\in R^{h,\kappa}. 
\end{equation}
Notice that $p^K_{R\oplus R^{h,\kappa}}=q^K_{R\oplus R^{h,\kappa}}=m^K_{R}=m^K_{R^{h,\kappa}}.$

\begin{definition}[{\cite[Definition 15.5]{ALTV}}] The Grothendieck group of admissible $\kappa$-hermitian $\CH$-modules is the abelian group $R(\CH)^\kappa$ generated by $(V,\langle~,~\rangle_V)$, where $V$ is a smooth admissible $\CH$-module with the nondegenerate $\kappa$-invariant form $\langle~,~\rangle_V$ subject to the following relations: whenever $U$ is an $\CH$-submodule of $V$ and $R$ is the radical of the form $\langle~,~\rangle_U$ obtained by restricting $\langle~,~\rangle_V$ to $U$, we have
\[ [V,\langle~,~\rangle_V]=[U/R,\langle~,~\rangle_{U/R}]+[U^\perp/R,\langle~,~\rangle_{U^\perp/R}]+[R\oplus R^{h,\kappa},\langle~,~\rangle_{hyp}].
\]
For every compact open subgroup $K$, we have a well defined signature homomorphism
\[\sg^K_\bullet=(p^K_\bullet,q^K_\bullet):R(\CH)^\kappa\to Fun[\widehat K, \bZ^2],\quad [V,\langle~,~\rangle_V]\mapsto (p_V^K,q_V^K),
\]
and a multiplicity homomorphism
\[m^K_\bullet :R(\CH)\to Fun[\widehat K, \bZ],\quad [V]\mapsto m^K_V.
\]
Here $R(\CH)$ denotes the ordinary Grothendieck group of admissible $\CH$-modules, and $Fun[~]$ denotes the set of functions.
\end{definition}
As in \cite{ALTV}, a better way to define the signature homomorphism is by introducing the signature Grothendieck ring $\bW$ of hermitian finite dimensional $\bC$-vector spaces. This is the hermitian Grothendieck group of finite dimensional $\bC$-vector spaces together with the tensor product. The unit is $1=[\bC,\langle~,~\rangle_{std}]$, where $\langle~,~\rangle_{std}$ is the standard positive definite form, and if we set $s=[\bC,-\langle~,~\rangle_{std}]$, then
\[\bW=\bZ[s]/\<s^2-1\>.
\]
Consider the tensor product of $\CH$-modules with finite dimensional vector spaces (endowed with the trivial $\CH$-action). This can also be extended in the obvious way to define a structure of a $\bW$-module on $R(\CH)^\kappa$. It is not hard to see (\cite[Proposition 15.10]{ALTV}) that $R(\CH)^\kappa$ is generated by

\begin{enumerate}
\item $[L,\<~,~\>_L]$, where $L$ ranges over the irreducible hermitian $\CH$-modules ($L\cong L^{h,\kappa}$) with a fixed choice of nondegenerate invariant form $\<~,~\>_L$;
\item $[L'\oplus {L'}^{h,\kappa},\<~,~\>_{hyp}]$, where $L'\not\cong{L'}^{h,\kappa}$ range over the unordered pairs of irreducible non-hermitian modules. Clearly,
\[s [L'\oplus {L'}^{h,\kappa},\<~,~\>_{hyp}]=[L'\oplus {L'}^{h,\kappa},\<~,~\>_{hyp}] \text{ in } R(\CH)^\kappa.
\]
\end{enumerate}
The following result expresses the hermitian form of an admissible module in terms of the hermitian forms of its composition factors.
\begin{proposition}[{\cite[Proposition 15.10]{ALTV}}]\label{p:herm-Groth} Every admissible hermitian module $V$ can be written uniquely in $R(\CH)^\kappa$ as
\[[V,\<~,~\>_V]=\sum_{L\cong L^{h,\kappa}} w(L,V) [L,\<~,~\>_L]+\sum_{L'\not\cong {L'}^{h,\kappa}} m(L',V) [L'\oplus {L'}^{h,\kappa},\<~,~\>_{hyp}], 
\]
where 
\begin{equation*}
\begin{aligned}
m(L',V)&=\text{multiplicity of }L'\text{ in }V\text{ as a composition factor},\\
w(L,V)&=p(L,V)+q(L,V) s\in\bW,\ p(L,V),q(L,V)\ge 0, \\ 
m(L,V)&=p(L,V)+q(L,V).
\end{aligned}
\end{equation*}
In particular, for every compact open subgroup $K$, the $K$-signature character of $V$ is
\begin{equation}
\sg_V^K=\sum_{L\cong L^{h,\kappa}} w(L,V) \sg_L^K+(1+s)\sum_{L'\not\cong {L'}^{h,\kappa}} m(L',V) m^K_{L'}.
\end{equation}
\end{proposition}
In this notation, a nondegenerate hermitian form $\<~,\>$ on $V$ is positive definite fi the coefficient of $s$ in the signature is zero.

\subsection{} Fix a minimal parabolic subgroup $\C P_0=\C M_0 \C N_0$ defined over $F$. Let $\C A_0\subset \C M_0$ be a maximal split torus over $F$. We will denote by $P_0$ the $F$-points of $\C P_0$, by $M_0$ the $F$-points of $\C M_0$ etc. As in \cite[\S XI.1]{BW}, a standard pair $(P,A)$ is defined by a standard parabolic $\C P\supset \C P_0$, $\C P=\C M\C N$, and a maximal $F$-split torus in the center of $\C M$ such that $\C A\subset \C A_0$. Denote $X(\C A)_F=\Hom_F[\C A,\mathbb G_m]$ and set
\[\fa^*=X(\C A)_F\otimes_{\bZ}\bR,\ \fa^*_c=\fa^*\otimes_\bR \bC, \quad \fa=\Hom_{\bZ}[X(\C A),\bR].
\]
For a standard pair $(P,A)$, we regard $\fa^*$ as a real subspace of $\fa_0^*$ and similarly $\fa\subset \fa_0$. Let $H:M\to \fa$ be the Harish-Chandra map \cite[\S XI.1.13]{BW}. Choose an inner product $(~,~)$ on $\fa_0^*$ which is invariant under the action of the Weyl group of $A_0$ in $G$. Let $\Phi(\C P,\C A)$ denote the root system of $(\C P,\C A)$. Define
\[\fa^{*,+}=\{\nu\in\fa\mid (\nu,\alpha)>0\text{ for all }\alpha\in \Phi(\C P,\C A)\}.
\]
Let $\{\alpha_i\}$ be the set of simple roots in $\Phi(\C P,\C A)$ and let $\{\omega_i\}$ be a basis $(~,~)$-dual to  $\{\alpha_i\}$ in the span of $\{\alpha_i\}$. Define a partial order $\le$ on $\fa^*$, by setting $\nu\le\lambda$ whenever $(\lambda-\nu,\omega_i)\ge 0$ for all $i$.

\smallskip

For a standard parabolic subgroup  $P=MN$ of $G$, denote by $i_P^G$ the functor of normalized parabolic induction and by $r_P^G$ the normalized Jacquet functor (the normalized functor of $N$-coinvariants). Let $M^0$ be the subgroup generated by all compact subgroups of $M$, an open normal subgroup of $M$. Let $\delta_P$ be the modulus function. Denote by $R(G)$ the Grothendieck group of admissible $G$-representations. We may regard $i_P^G$ as a map $i_P^G:R(M)\to R(G)$, independent of $P\supset M$, and therefore it may be denoted by $i_M^G$ in this situation. 

If $(\sigma,U_\sigma)$ is an admissible $M$-representation, and $\nu\in\fa^*_c$, define the induced representation:
\begin{equation}\label{ind-rep}
\begin{aligned}
&I(P,\sigma,\nu)=i_P(\sigma\otimes\nu)\\
&=\{f:G\to U_\sigma\mid f \text{ loc. const.},\ f(mng)=\delta_P^{1/2}(m)\sigma(m) q^{\nu(H(m))} f(g),\ mn\in P\},
\end{aligned}
\end{equation}
with the $G$-action given by right translations.

A Langlands datum is a triple $(P,\sigma,\nu)$, where $(P,A)$ is a standard pair, $P=MN$, $\sigma$ is an irreducible tempered representation of $M^0$, and $\nu\in\fa^{*,+}$. The following theorem is the Langlands classification in the $p$-adic case, and summarizes \cite[Proposition 2.6, Theorems 2.10 and 2.11, Lemma 2.13]{BW} and \cite[Theorem 4.1]{Si}, cf. \cite[Theorem 5.1]{BM}.

\begin{theorem}\label{t:Langlands}
 Let $(P,\sigma,\nu)$ be a Langlands datum and $\bar P=M\bar N$ the opposite parabolic to $P$.

\begin{enumerate}
\item The integral operator
\begin{equation}
j(\nu): I(P,\sigma,\nu)\to I(\bar P,\sigma,\nu),\quad j(\nu)(f)(g)=\int_{\bar N} f(\bar n g) ~d\bar n,
\end{equation}
converges absolutely and uniformly in $g$ on compacta, and it is an intertwining operator with respect to the $G$-actions.
\item $J(P,\sigma,\nu)=j(\nu)(I(P,\sigma,\nu))=I(P,\sigma,\nu)/\ker j(\nu)$ is the unique irreducible quotient of $I(P,\sigma,\nu)$, called the Langlands quotient.
\item Let $(P',\sigma',\nu')$ be another Langlands datum. If $J(P,\sigma,\nu)\cong J(P',\sigma',\nu')$, then $P=P'$, $\sigma=\sigma'$, and $\nu=\nu'$.
\item Let $\pi$ be an irreducible smooth $G$-representation. Then there exists a Langlands datum $(P',\sigma',\nu')$ such that $\pi\cong J(P',\sigma',\nu')$. The parameter $\nu'\in \fa^{*,+}$ is called the Langlands parameter of $\pi$ and it is denoted by $\lambda_\pi$. 
\item Suppose $\tau$ is an irreducible composition factor of $I(P,\sigma,\nu)$ and $\pi=J(P,\sigma,\nu)$. Then $\lambda_\tau\le \lambda_\pi$, and equality occurs if and only if $\tau=\pi$.
\end{enumerate}

\end{theorem}

\subsection{} From now on, we restrict to the case of hermitian representations with respect to the natural star operation
$\kappa(f)=f^*.$ The classification of $*$-hermitian irreducible smooth $G$-representations is well known. The $p$-adic case, see for example \cite[\S 5]{BM}, is a direct analogue of the results of Knapp and Zuckermann for $(\fg,K)$-modules of real reductive groups \cite[Chapter XVI]{Kn}.

Let $(P,\sigma,\nu)$ be a Langlands datum. The Langlands quotient $J(P,\sigma,\nu)$ admits a nondegenerate $*$-hermitian form if and only if there exists $w\in W(G,A)$, the Weyl group of $A$ in $G$, such that
\begin{equation}
w(M)=M,\ ^w\sigma\cong \sigma \text{ and } w(\nu)=-\nu.
\end{equation}
Here $w(M)=\dot w^{-1} M \dot w$ for some (fixed) lift $\dot w$ of $w$ in the normalizer $N_G(A)$, and $^w\sigma(m)=\sigma(\dot w m \dot w^{-1})$. Suppose this is the case. Fix an isomorphism
\[a_{\sigma,w}: \sigma\to~ ^w\sigma,
\]
which induces an isomorphism, still denoted $a_{\sigma,w}$ between $I(P,\sigma,-\nu)$ and $I(\bar P,\sigma,\nu)$.
We emphasize that $a_{\sigma,w}$ does not depend on $\nu$. Let $\<~,~\>_\sigma$ be a positive definite $*_M$-invariant hermitian form on $U_\sigma$. Define a possible degenerate $*$-invariant hermitian form on $I(P,\sigma,\nu)$ by
\begin{equation}
\<f_1,f_2\>_{P,\sigma,\nu,w}=\int_G\<f_1(g), a_{\sigma,w}^{-1} (j(\nu)(f_2)(g))\>_\sigma ~dg.
\end{equation}
The radical of $\<~,~\>_{P,\sigma,\nu,w}$ is $\ker j(\nu)$. Taking the quotient by the radical induces an nondegenerate $*$-invariant hermitian form on $J(P,\sigma,\nu)$. Since $J(P,\sigma,\nu)$ is irreducible, we know that this form is unique up to scalar multiple.

\subsection{}\label{s:inf-char}
We recall the definitions of the Bernstein center and the inertial equivalence class.  Let $M$ be a Levi subgroup of $G$ with the subgroup $M^0$ as before. It is known that $M/M^0$ is a free abelian group of finite rank. A character $\chi:M\to \bC^\times$ is called unramified if $\chi(M^0)=1$. Write $\CX(M)$ for the group of unramified characters of $L$.

Consider the set of pairs $(M,\sigma)$, where $M$ is a Levi subgroup and $\sigma$ is a supercuspidal representation of $M$. Define the relation of equivalence such pairs: $(M_1,\sigma_1)\equiv (M_2,\sigma_2)$ if there exists $x\in G$ such that $M_2=~^x M_1$ and $\sigma_2\cong ~^x\sigma_1$. Denote by $\Theta(G)$ the set of equivalence classes. A point $\theta\in \Theta(G)$ is called an infinitesimal character of $G$.

\smallskip

Define the weaker relation of inertial equivalence on the same pairs: $(M_1,\sigma_1)\sim (M_2,\sigma_2)$ if there exists $x\in G$ such that $M_2=~^x M_1$ and $\sigma_2\cong ~^x\sigma_1\otimes\chi$ for some $\chi\in \CX(L_2)$.  Denote by $\CB(G)$ the set of inertial equivalence classes $[M,\sigma]$. Each $\fk s=[M,\sigma]\in \CB(G)$ defines a connected component of $\Theta(G)$ which will be denoted $\Theta(G)_{\fk s}$. It has the structure of a complex affine algebraic variety as the quotient of $\CX(M)$ by a finite group $W_{\fk s}$. The decomposition 
\begin{equation}\label{theta-decomp}
\Theta(G)=\cup_{\fk s\in\CB(G)}\Theta(G)_{\fk s},
\end{equation}
makes $\Theta(G)$ into a complex algebraic variety with infinitely many connected components.
\medskip

As it is well known, if $(\pi,V)\in\Irr G$, there exists a unique up to equivalence cuspidal pair $(L,\sigma)$ such that $\pi$ is a subquotient of $i_P^G(\sigma).$ Then $(L,\sigma)$ defines a point $\theta\in \Theta(G)$, which is called the infinitesimal character of $\pi$. The infinitesimal character map
\begin{equation}
\text{inf-char}: \Irr G\to \Theta(G)
\end{equation}
is onto and finite to one. 

The Bernstein center $\mathfrak Z(G)$ is the algebra of regular functions on $\Theta(G)$. From (\ref{theta-decomp}), one decomposed $\mathfrak Z(G)=\prod_{\fk s} \mathfrak Z(\fk s)$, where $\mathfrak Z(\fk s)$ is the algebra of regular functions on the Bernstein component $\fk s$. Let $\mathfrak Z^0(G)=\bigoplus_{\mathfrak s} \mathfrak Z(\fk s)$ be the ideal of $\mathfrak Z(G)$ of functions supported on a finite number of components. Then each infinitesimal character $\theta\in\Theta(G)$ can be identified with an algebra homomorphism $\theta:\mathfrak Z(G)\to \bC$ which is nontrivial on $\mathfrak Z^0(G)$. With this identification, on each irreducible smooth representation $(\pi,V)$, $z\in \mathfrak Z(G)$ acts by (see \cite[\S2.13]{BD})
\[ \text{inf-char}(\pi)(z)\cdot \Id_V.
\]
For the Levi subgroup $M$ of a standard parabolic subgroup $P$, let $\Theta(M)$ denote the corresponding variety of infinitesimal characters. Let $i_{M}^G:\Theta(M)\to \Theta(G)$ denote the corresponding finite morphism on algebraic varieties, see \cite[\S2.4]{BDK}.

\subsection{}We call the inertial class $\fk s=[L,\sigma]$ the inertial support of $\pi$. Let $\C C^{\fk s}(G)$ denote the full subcategory of $\C C(G)$ consisting of representations $\pi$ with the property that every subquotient of $\pi$ has inertial support $\mathfrak s$. Then 
\begin{equation}
\C C(G)=\prod_{\fk s\in \CB(G)}\C C^{\fk s}(G),
\end{equation}
as abelian categories. If $e\in\CH$ is a special idempotent, then the results of \cite[3.12]{BK} (see also \cite[3.6]{BHK}) say that there exists a subset $\C S(e)\subset \CB(G)$ such that
\begin{equation}
\C C_e(G)=\prod_{\fk s\in \C S(e)} \C C^{\fk s}(G).
\end{equation}
Conversely, it is shown in {\it loc. cit.} that for every $\mathfrak s\in \CB(G)$, there exists a special, self-adjoint idempotent $e$ such that $\C C_e(G)=\C C^{\fk s}(G)$.

\subsection{}Let $K$ be a compact open subgroup. The goal is to express the $K$-signature of an irreducible smooth $G$-representation in terms of that of tempered modules. A first step is the simpler result for $K$-multiplicity functions.

\begin{lemma}\label{l:K-char} Let $(\pi,V)$ be an irreducible smooth $G$-representation in $\C C^{\fk s}(G)$. There exist irreducible tempered $G$-representations $\{V_i\mid 1\le i\le n\}$ in $\C C^{\fk s}(G)$ and integers $a_i$ such that for every $K$,
\[m_V^k=\sum_{i=1}^n a_i m_{V_i}^K.
\]
\end{lemma}

\begin{proof}
Set $\theta=\text{inf-char}(\pi)\in\Theta(G)$. There are finitely many irreducible $G$-representations, all in $\C C^{\fk s}(G)$, with infinitesimal character $\theta$. Write them in terms of the Langlands classification as $J(P,\sigma,\nu)$ and fix an ordering on this set compatible with $\le$ on the Langlands parameters $\nu$.

Suppose $\pi_1=J(P_1,\sigma_1,\nu_1)$ is a minimal element in this order. By Theorem \ref{t:Langlands}(5), we must have $J(P_1,\sigma_1,\nu_1)=I(P_1,\sigma_1,\nu_1)$, hence $m_{\pi_1}^K=m_{I(P_1,\sigma_1,\nu_1)}^K=m_{I(P_1,\sigma_1,0)}^K$, and $I(P_1,\sigma_1,0)$ is tempered and in $\C C^{\fk s}(G)$. If $I(P_1,\sigma_1,0)$ is reducible then it decomposes into a direct sum of finitely many irreducible tempered representations, hence the claim is proved in the base case of the induction.

Now let $V=J(P,\sigma,\nu)$ be arbitrary. If $J(P,\sigma,\nu)=I(P,\sigma,\nu)$, the claim follows in the same way as before. Otherwise, let $J_i=J(P_i,\sigma_i,\nu_i)$ be the composition factors of $I(P,\sigma,\nu)$ other than $J(P,\sigma,\nu)$. Then
\[m_V^K=m_{I(P,\sigma,\nu)}^K-\sum_i m(J_i,I(P,\sigma,\nu)) ~m_{J_i}^K.
\]
Since $m_{I(P,\sigma,\nu)}^K=m_{I(P,\sigma,0)}^K$ and $\lambda_{J_i}<\lambda_V$, the proof is complete by induction.
\end{proof}

\subsection{} We summarize the necessary results about the Jantzen filtration and the signature character, see \cite[\S3]{Vo}, \cite[\S14]{ALTV}, also \cite[\S5]{BM} for the $p$-adic analogue.

Let $(P,\sigma,\nu)$ be a Langlands datum such that $J(P,\sigma,\nu)$ admits a nondegenerate $*$-invariant hermitian form. Regard the standard module $I(P,\sigma,\nu)$ as a member of the continuous family $I_t=I(P,\sigma, t\nu)$, $t\in \bR_{\ge 0}$. As it is well known, all of the representations in the family $I_t$ can be realized on the same space $X$. Explicitly, let $K_0$ be a maximal compact open subgroup such that the Iwasawa decomposition holds $G=PK_0$. Then 
\begin{equation}
X=\{f:K_0\to U_\sigma\mid f\text{ loc. const.},\ f(mnk)=\sigma(m) f(k),\ mn\in P\cap K_0,\  k\in K_0\}.
\end{equation} 
It is important to notice that for every compact subgroup $K$, the restriction of $I_t$ to $K$ is independent of $t$.

The hermitian forms $\<~,~\>_{P,\sigma,t\nu,w}$ on $I_t$ can be regarded as an analytic family of hermitian forms $\<~,~\>_t$ on $X$ such that 
\[J(P,\sigma,t\nu)=I(P,\sigma,t\nu)/\ker\<~,~\>_t,\text{ for }t>0.
\]
The Jantzen filtration of $I_1$ ($t=1$, obviously the same definition applies to an arbitrary point $t_0>0$) is:
\begin{equation}
X=X^0\supset X^1\supset X^2\subset\dots\supset X^N=\{0\},
\end{equation}
where $X^n$ consists of the vectors $x\in X$ such that there exists a an analytic function 
\[ f_x:(1-\epsilon,1+\epsilon)\to X,
\]
satisfying the conditions:
\begin{enumerate}
\item[(a)] $f_x$ takes values in a fixed finite dimensional space of $X$;
\item[(b)] $f_x(1)=x$;
\item[(c)] for all $y\in X$, the function $t\mapsto \<f_x(t),y\>_t$ vanishes at least to order $n$ at $t=1$. 
\end{enumerate}
Define a hermitian form $\<~,~\>^n$ on $X^n$ by
\begin{equation}
\<x,x'\>^n=\lim_{t\to 1}\frac 1{(t-1)^n}\<f_x(t), f_{x'}(t)\>_t.
\end{equation}
The main result is the following theorem.

\begin{theorem}[{\cite[Theorem 3.8]{Vo}}]\label{t:Jantzen}
With the notation as above, the Jantzen filtration $X=X^0\supset X^1\supset X^2\supset\dots\supset X^N=\{0\}$ is a filtration of $G$-representations.
\begin{enumerate}
\item[(a)] $X^0/X^1$ is the Langlands quotient $J_1$.
\item[(b)] The radical of the form $\<~,~\>^n$ is $\<~,~\>^{n+1}$. The nondegenerate form $\<~,~\>^n$ on $X^n/X^{n+1}$ is $*$-invariant. Write $\sg_1^{K,n}=(p_n,q_n): \widehat K\to \bN\times \bN$ for its $K$-signature.
\item[(c)] For small $\epsilon$, the nondegenerate invariant hermitian form $\<~,~\>_{1+\epsilon}$ on $I_{1+\epsilon}$ has signature
\[\left(\sum_n p_n,\sum_n q_n\right).
\] 
\item[(d)] For small $\epsilon$, the nondegenerate invariant hermitian form $\<~,~\>_{1-\epsilon}$ on $I_{1-\epsilon}$ has signature 
\[\left(\sum_{n \text{ even}} p_n+\sum_{n \text{ odd}} q_n,\sum_{n \text{ odd}} p_n+\sum_{n \text{ even}} q_n\right).
\]
\end{enumerate}
Equivalently,
\begin{equation}\label{e:cross}
\sg_{1+\epsilon}^K=\sg_{1-\epsilon}^K+(1-s)\sum_{n\text{ odd}} \sg_1^{K,n},
\end{equation}
where $\sg_{t}$ denotes the $K$-signature of $\<~,~\>_t$.
\end{theorem}
Before we get to the signature theorem, we need a lemma which will insure that the inductive algorithm terminates (see also \cite[Lemma 2.12]{Sol} for similar considerations).

\begin{lemma}\label{l:terminate}
Let $\fk s$ be a Bernstein component. There exists a positive real number $d_{\fk s}$ (depending only on $\fk s$) with the following property: if $(P,\sigma,\nu)$ is a Langlands datum and $V=J(P,\sigma,\nu)$ is the Langlands quotient, for every other irreducible composition factor $V'\neq V$ of $I(P,\sigma,\nu)$ we have $||\lambda_V-\lambda_{V'}||>d_{\fk s}$. Here the norm $||~||$ is taken with respect to $(~,~)$, the inner product on $\fk a_0^*$.
\end{lemma}

\begin{proof}
Suppose $V'=J(P',\sigma',\nu')$ in the Langlands classification. Since $V$ and $V'$ are both composition factors of $I(P,\sigma,\nu)$, we have $\text{inf-char}(V)=\text{inf-char}(V')$. On the other hand, by \cite[Theorem 2.8.1]{Si2}, there exists an irreducible discrete series representation $\sigma_0$ of a Levi subgroup $S\subset M$ and a unitary unramified character $\chi_0$ of the center of $S$ such that $\sigma$ is a direct summand of the unitarily induced representation $i_S^M(\sigma_0\otimes\chi_0)$. Similarly, we have $\sigma_0'$, $\chi_0'$ for $\sigma$.

This means that 
\begin{equation}\label{e:cc}
i_S^G(\text{inf-char}(\sigma_0))\chi_0 q^\nu=\text{inf-char}(V)=\text{inf-char}(V')=i_{S'}^G(\text{inf-char}(\sigma_0'))\chi_0' q^{\nu'},
\end{equation} 
 and $\nu'<\nu$. By \cite[Proposition 3.1]{BDK}, in every Bernstein component there are only finitely many discrete infinitesimal characters (modulo the action of the unramified characters). The infinitesimal characters of irreducible discrete series are discrete, hence there are only  finitely many orbits of infinitesimal characters
\begin{equation}
\{i_S^G(\text{inf-char}(\sigma_0))\mid S\text{ Levi of }G,\ \sigma_0\text{ discrete series of }S\}\cap \Theta_{\fk s},
\end{equation}
under the action of the unitary unramified characters. Since $\nu,\nu'$ are in $\fa_0^*$ (i.e., they are real unramified characters), the claim follows from (\ref{e:cc}).
\end{proof}

We can now state the analogue of Vogan's signature theorem, cf. \cite[Theorem 1.5]{Vo}, \cite[Theorem 5.2]{BM}. 
\begin{theorem}[The signature theorem]\label{t:signature} Let $(\pi,V)$ be an irreducible smooth $G$-representation with inertial support $\fk s$. Suppose that $V$ admits a $*$-invariant Hermitian form. Then there exist irreducible tempered representations $V_1,\dots,V_n$ in $\C C^{\fk s}(G)$ and $w_1,\dots,w_n\in\bW$ such that, for every compact open subgroup $K$, the $K$-signature of $V$ is
\[\sg_V^K=\sum_{i=1}^n w_i \sg_{V_i}^K.
\]
\end{theorem}

\begin{proof}The proof is the same as in {\it loc. cit.}, except that we need to explain in our setting the independence of $K$, the Bernstein component $\fk s$, and the fact that the algorithm terminates.

Write $V$ in terms of the Langlands classification as $V=J(P,\sigma,\nu)$. As before, consider $I_t=I(P,\sigma,t\nu)$, $t\in [0,1]$. Let $t_1<\dots<t_{r_1}$ be the points in $(0,1)$ where $I_t$ is reducible, and set $t_0=0$ and $t_r=1$. For $1\le j\le r$ define
\[\sg^{K}_j=\text{ the $K$-signature of }\<~,~\>_t,\ t\in(t_{j-1},t_j).
\]
For $1\le j\le r$, the induced representation $I_{t_j}$ has the Jantzen filtration
\[X_{t_j}=X^0_{t_j}\supset X_{t_j}^1\supset \dots.
\]
The subquotient $X_{t_j}^n/X_{t_j}^{n+1}$ has the nondegenerate form $\<~,~\>_n^j$ with signature $\sg^{K,n}_j$. By Theorem \ref{t:Jantzen}, we have the recursion formulas for $\sg^K_j$:
\begin{equation}
\begin{aligned}
&\sg^K_{j+1}=\sum_n \sg^{K,n}_j,\\
&\sg^K_{j+1}=\sg^K_{j}+(1-s)\sum_{n\text{ odd}}\sg_j^{K,n}.
\end{aligned}
\end{equation}
In particular, the signature of $V$ is:
\begin{equation}
\sg^K_V=\sg_r^0=\sg^K_r+(1-s)\sum_{n\text{ odd}}\sg_r^{K,n}-\sum_{m>0}\sg_r^{K,m}.
\end{equation}
As in \cite[3.38]{Vo}, this leads to
\begin{equation}\label{e:induction}
\sg^K_V=\sg^K_1+(1-s)\sum_{j=1}^{r}\sum_{n\text{ odd}}\sg^{K,n}_j-\sum_{m>0}\sg_r^{K,m}.
\end{equation}
Notice first that in the right hand side of this formula, the signature of $J_{t_j}$ does not appear for any $j\ge 1.$ The signature $\sg^K_1$ is the same as the signature of $I(P,\sigma,0)$, which is a possibly reducible tempered module, and its signature can be written uniquely in terms of the signatures of its irreducible tempered composition factors. For the rest of the terms, we want to proceed by induction on the length of the Langlands parameter $\nu=\lambda_V$. 

One remarks, as in \cite[Lemma 3.1]{Vo} that if the conclusion of Theorem \ref{t:signature} holds for all the irreducible hermitian composition factors of an admissible module $Y$ which have a nondegenerate hermitian form, then it holds for $Y$ itself. This is because of Proposition \ref{p:herm-Groth} and Lemma \ref{l:K-char}. In particular, we may apply this observation to our situation, namely to the (possibly reducible) subquotients $X_{t_j}^n/X_{t_j}^{n+1}$ in the Jantzen filtrations, as long as by induction we may assume that the claim holds for the irreducible composition factors of $I_{t_j}$, other than $J_{t_j}$.

Let $V'$ be an irreducible hermitian composition factor of $I(P,\sigma,t_j\nu)$ for some $1\le j\le r$, such that $V'\neq J(P,\sigma,t_j\nu)$. These are the irreducible representations whose signature contribute in the right hand side of (\ref{e:induction}). By Lemma \ref{l:terminate},  the Langlands parameter of $V'$ satisfies $||\lambda_{V'}||<||t\nu_j||-d_{\fk s}\le ||\lambda_V||-d_{\fk s}<||\lambda_V||.$ The induction hypothesis applies then and therefore the conclusion of the theorem holds for $V$ as well, by (\ref{e:induction}). Notice that the algorithm terminates since every time the length of the Langlands parameter drops by at least $d_{\fk s}>0$.

\end{proof} 

\begin{corollary}
With the same notation as in Theorem \ref{t:signature}:
\[\sg_V^K=\sum_{i=1}^n w_i m_{V_i}^K.
\]
\end{corollary}

\begin{proof}
Since every $V_i$ is tempered, hence $*$-unitary, the signature of $V_i$ is $\sg_{V_i}^K=m_{V_i}^K.$
\end{proof}

\begin{remark}The same formula in Theorem \ref{t:signature} holds with respect to different star operations $\kappa$ as long as for every $\kappa$-hermitian Langlands quotient $J(P,\sigma,\nu)$, we can define a natural (degenerate) $\kappa$-invariant form on $I(P,\sigma,\nu)$ whose radical is $\ker j(\nu)$. In the setting of $(\fg,K)$-modules, an essential role is played by the signature theorem for the compact star operation \cite{ALTV}. We will explore a $p$-adic analogue in future work.
\end{remark}

\section{Rigid tempered representations}\label{s:rigid}
In this section, we explain several relevant results from \cite{BDK}, \cite{Da,Da2}, and \cite{CH2} which will allow us to sharpen the signature theorem \ref{t:signature}. See also \cite{Bez}, \cite{SS}, \cite{Vi} for more details on the relation between the cocenter, $K_0(G)$, and characters of admissible representations.

If $\C A$ is an abelian group and $R$ is a ring, we will denote $\C A_R=\C A\otimes_\bZ R.$

\subsection{}Recall that $R(G)$ denotes the $\bZ$-span of the set of irreducible smooth $G$-representations. 
The cocenter of $G$  (or of $\CH(G)$) is the linear space
\begin{equation}
\overline\CH(G)=\CH(G)/[\CH(G),\CH(G)],
\end{equation}
where $[\CH(G),\CH(G)]$ is the linear subspace spanned by all the commutators. An alternative description of $[\CH(G),\CH(G)]$ is that it is the subspace spanned by all $f-~^xf$, where $f\in\CH(G)$ and $x\in G$, $^xf(g)=f(x^{-1}g x).$ 

A distribution on $G$ is a linear map $D:\CH(G)\to \bC$. Let $\C D(G)$ denote the space of $G$-distributions. A $G$-invariant distribution, i.e., an element of $\C D(G)^G$, is the same as a linear map $D:\overline\CH(G)\to\bC$.

Let $G_c^0$ denote the open set of compact elements of $G$ and let $\CH_c(G)\subset \CH(G)$ be the subalgebra of functions supported in $G_c^0$. Define the compact cocenter:
\begin{equation}
\overline\CH_c(G)=\text{ the image of }\CH_c(G)\text{ in }\overline\CH_c(G).
\end{equation}
In other words, $\overline\CH_c(G)$ is the image in $\overline\CH(G)$ of $\bigoplus_K \CH(K)$, where $K$ ranges over the compact open subgroups of $G$ (equivalently, over the $G$-conjugacy classes of maximal compact open subgroups of $G$). 
The Hecke algebra $\CH(G)$ is the direct limit 
\[\CH(G)=\varinjlim_K \CH(G,K)
\]
of the algebras $\CH(G,K)$ of $K$-biinvariant functions, where $K$ ranges over a system of good compact open subgroups (in the sense of \cite{BD}). For example, one may take the system $\{I_m\}$, given by the $m$-th congruence subgroups of a fixed Iwahori subgroup $I$\footnote{In the notation of the introduction, $I_m$ is the Moy-Prasad filtration subgroup $G_{x,m^+}$ attached to a generic point $x$ in the Bruhat-Tits building.}. 

\subsection{} Define $R_\ind(G)_\bQ\subset R(G)_\bQ$ to be the $\bQ$-span of all induced modules $i_P^G(\sigma)$, $\sigma\in R(L)_\bQ$, where $P\neq G$. 
\begin{definition}[cf. \cite{CH2}]
Let $R_\diff(G)_\bQ\subset R(G)_\bQ$ denote the span of all differences $i_P^G(\sigma)-i_P^G(\sigma\otimes\chi)$, where $P$ ranges over the set of parabolic subgroups, $\sigma\in R(L)$, and $\chi\in \CX(L)$.
Define the compact representations space (quotient):
\begin{equation}
\overline R_c(G)_\bQ=R(G)_\bQ/R_\diff(G)_\bQ.
\end{equation}
\end{definition}
In the case when $G$ is semisimple, this is the same notion as that of the rigid quotient from \cite{CH2}. 

\begin{lemma}
The space $\overline R_c(G)_\bQ$ is spanned by the classes of irreducible tempered $G$-representations.
\end{lemma}
\begin{proof}
By the Langlands classification, $R(G)$ is spanned by the Langlands standard modules $I(P,\sigma,\nu)$, where $(P,\sigma,\nu)$ are Langlands date. It is clear that in $\overline R_\rig(G)_\bQ$, we have $I(P,\sigma,\nu)\equiv I(P,\sigma,0)$, the latter being a tempered representation of $G$.
\end{proof}

The elliptic representation space is
\[\overline R_\el(G)=R(G)/R_\ind(G).
\]
The space $R_\el(G)_\bQ$ is also spanned by the images of the irreducible tempered modules (again by invoking the Langlands classification), and there is a natural surjection $\overline R_\rig(G)_\bQ\to \overline R_\el(G)_\bQ$. Every irreducible discrete series representation gives a nonzero class in $R_\el(G)_\bQ$, see for example \cite{Ka2}.

\subsection{}Let $K_0(G)$ denote the Grothendieck group of finitely generated projective $G$-representations. Let
\begin{equation}
\tau: K_0(G)\to \overline\CH(G)
\end{equation}
denote the Hattori-Stallings trace map, see for example \cite[\S2.2]{Da2}.
The following result is essentially the abstract Selberg principle for $p$-adic groups proved in \cite{BB}.
\begin{theorem}[{see \cite[Theorem 1.6]{Da}}]\label{t:HS}The image of $K_0(G)_\bC$ under $\tau$ is $\overline\CH_c(G)$. 
\end{theorem}
More is known about $\tau$, for example, it is also proved in the \cite{Da} that
\begin{enumerate}
\item[(a)] The map $\tau$ is injective, and 
\item[(b)] The space $K_0(G)_\bQ$ is generated over $\bQ$ by compactly induced modules
\[\cind_K^G(\mu)=\CH(G)\otimes_{\CH(K)}E_\mu,
\]
where $K$ ranges over the compact open subgroups of $G$ and $(\mu,E_\mu)$ over $\widehat K$.
\end{enumerate}
We will not need these finer results.

\subsection{}As in \cite[\S1.2]{BDK}, a linear map $F:R(G)\to \bC$ is called a good form if
\begin{enumerate}
\item[(i)] For every standard Levi subgroup $M<G$, and every irreducible smooth $M$-representation $\sigma$, the function $\psi\mapsto i_M^G(\sigma\otimes\psi)$ is a regular function on the complex algebraic variety $\C X(M)$.
\item[(ii)] There exists a compact open subgroup $K$ such that $f(V)=0$ for all admissible representations $(\pi,V)$ such that $V^K=0$.
\end{enumerate}
Denote by $\C F(R(G))$ the space of linear maps $R(G)\to\bC$ and by $\C F(R(G))_\good\subset \C F(R(G))$ the space of good forms on $R(G)$. 
The following is the classical trace Paley-Wiener theorem.
\begin{theorem}[{\cite[Theorem 1.2]{BDK}}]
The trace map $\tr: \overline \CH(G)\to \C F(R(G))$ given by $f\mapsto (\pi\mapsto \tr\pi(f))$ is surjective onto $\C F(G)_\good$. 
\end{theorem}
In fact, Kazhdan's density theorem says that $\tr$ is also injective \cite{Ka1,Ka2}, but we will not use this result.

\subsection{}We now look at the restriction of $\tr$ to $\overline \CH_c(G)$.
If $\sigma\in\widehat K$, $K$ a compact open subgroup, let $\chi_\sigma:K\to \bC$ denote its character. One may extend this by zero to an element of $\CH(G)$ and regard $\chi_\sigma$ as an element of $\overline \CH_c(G).$ In fact, the compact cocenter $\overline\CH_c(G)$ is spanned by $\chi_\sigma$ as $\sigma$ ranges over $\widehat K$ for all compact open subgroups $K$.
If $\pi$ is an admissible $G$-representation, then it is well known that
\[\tr\pi(\chi_\sigma)=\mu(K) ~m^K_\pi(\sigma^\vee),\] 
where $\sigma^\vee$ is the contragredient representation of $\sigma$.
Therefore, if $f\in \overline \CH_c(G)$ then $\tr(\pi(f))=0$ for all $\pi\in R_\diff(G)$, since $i_P^G(\sigma)|_K=i_P^G(\sigma\otimes\chi)|_K$ for all compact $K$. This means that we naturally have a map
\begin{equation}
\tr_c:\overline\CH_c(G)\to \C F(\overline R_c(G))_\good.
\end{equation}
The following result is proved in \cite{CH2} in the more general setting of mod-$l$ representations. It  can also be deduced from \cite{BDK} in combination with \cite{Da,Da2}.
\begin{theorem}[The rigid trace Paley-Wiener theorem, \cite{CH2}]\label{t:rigid-PW}
The map $\tr_c:\overline\CH_c(G)\to \C F(\overline R_c(G))_\good$ is surjective. Moreover, for every Iwahori filtration subgroup $I_m$, $m>0$:
\begin{enumerate}
\item $\tr_c:\overline\CH_c(G,I_m)\to \C F(\overline R_c(G,I_m))$ is surjective.
\item $\dim\overline\CH_c(G,I_m)<\infty.$
\end{enumerate}
\end{theorem}

\subsection{}We need to look at the action of the Bernstein center. Let $U(G)$ denote the inverse limit algebra
\begin{equation}
U(G)=\varprojlim_K \CH(G,K),
\end{equation}
where for $K'\subset K$, the inverse limit system is given by $p_{K',K}: \CH(G,K')\to \CH(G,K),$ $p_{K',K}(f)=e_{K}f e_{K}.$ This is a unital associative algebra. As it is known \cite{BD}, see also the exposition in \cite{MT}, the algebra $U(G)$ has the equivalent description as the convolution algebra of essentially compact distributions:
\begin{equation}
U(G)=\{D\in\C D(G)\mid D\star f\text{ and } f\star D\text{ are in }\CH(G)\text{ for all }f\in\CH(G)\}.
\end{equation}
The Bernstein center can be identified as
\begin{equation}
\fk Z(G)=U(G)^G=Z(U(G)).
\end{equation}
The star operation $*$ extends to $U(G)$ by setting
\[
D^*(f)=D(f^*),\text{ for all }f\in\CH(G).
\]
The Hecke algebra itself is a $*$-subalgebra (in fact, a two-sided ideal) of $U(G)$. The Bernstein center $\fk Z(G)$ acts on $\CH(G)$ as the center of $U(G)$. This also gives an action of $\fk Z(G)$ on $\overline\CH(G)$. On the other hand, the central algebra $\fk Z(G)$ acts on $R(G)$ via infinitesimal characters (on each irreducible representation), see subsection \ref{s:inf-char}, and therefore on $\C F(G)$. As remarked in \cite[\S4.1]{BDK}, 
\[\tr:\overline\CH(G)\to \C F(G)_\good\text{ is a }\fk Z(G)\text{-homomorphism}.
\]
The similar refinement for the rigid trace Paley-Wiener theorem is more subtle. Let $\delta_{G_c^0}$ be the distribution on $\CH(G)$, $\delta_{G_c^0}(f)=f|_{G_c^0}$, i.e.,  the restriction to $G_c^0$.  Notice that $f\in \overline\CH_c(G)$ if and only if
\begin{equation}
 \delta_{G_c^0}(f)=f.
\end{equation}The nontrivial step is the following:
\begin{proposition}[{\cite[Proposition 2.8]{Da2}}]\label{p:Dat}
For every idempotent element $e\in \fk Z(G)$, 
\[e\circ \delta_{G_c^0}=\delta_{G_c^0}\circ e \text{ in }\End(\overline \CH(G))).
\]
\end{proposition}
Here, we think of $e\in \End(\overline \CH(G))$ via $e(f)=e\star f$, for $f\in\overline\CH(G)$. Proposition \ref{p:Dat} says that for every $f\in\overline\CH_c(G)$, 
\[e\star f=e(f)=e(\delta_{G_c^0}(f))=\delta_{G_c^0}(e(f)),
\]
which means that $\CH_c(G)$ is invariant under the actions of Bernstein idempotents.

For every Bernstein component $\fk s\in\C B(G)$, denote by $e_\fk s$ the corresponding idempotent. Set \[\CH(G,\fk s)=\CH(G)e_{\fk s},\ \overline\CH(G,\fk s)=\CH(G)e_{\fk s},\text{ and }R(G,\fk s)=R(G)e_{\fk s}.\]
 Let $\overline R_c(G,\fk s)$ be the image of $R(G,\fk s)$ in $\overline R_c(G)$.  By the discussion above, we can also define $\overline\CH_c(G,\fk s)=\overline\CH_c(G)e_{\fk s}\subset \overline\CH_c(G).$

\begin{corollary}\label{c:rigid}
The map $\tr_c:\overline \CH_c(G)\to \C F(\overline R_c(G))_\good$ induces surjective maps
\[\tr_c(\fk s): \overline \CH_c(G,\fk s)\to \C F(\overline R_c(G,\fk s)),
\]
for every Bernstein component $\fk s$. The spaces $\overline \CH_c(G,\fk s)$ are finite dimensional.
\end{corollary}

\begin{proof}
The first part is immediate from Theorem \ref{t:rigid-PW}(1) and Proposition \ref{p:Dat}. For the second part, notice that for each $\fk s$ there exists a sufficietly small subgroup $I_m$ such that $\CH(G,\fk s)\subset \CH(G,I_m)$. Then $\overline\CH_c(G,\fk s)\subset \overline\CH_c(G,I_m)$ and the second claim follows from Theorem \ref{t:rigid-PW}(2).
\end{proof}

\subsection{}Now suppose that $(\C K,\rho)$ is a type with $e_\rho\in \CH(G)$ the corresponding idempotent. Let $R(G,\rho)$ be the subspace of $R(G)$ spanned by the irreducible objects in $(\C K,\rho)$ and $\overline R_c(G,\rho)$ its image in $\overline R_c(G)$. Since $e_\rho$ is supported in $G_c^0$, $e_\rho \CH_c(G) e_\rho\subset \CH_c(G)$. Denote by $\overline{e_\rho \CH_c(G) e_\rho}$ the image of $e_\rho \CH_c(G) e_\rho\subset \CH_c(G)$ in $\overline \CH_c(G)$.

If $f\in e_\rho\CH e_\rho$ and $(\pi,V)$ is an admissible representation, then
\begin{equation}
\tr(f,V)=\tr(f,\pi(e_\rho)V),
\end{equation}
and, in particular, $\tr(f,V)=0$ for all irreducible $V$ not in $\C C_{e_\rho}(G)$. This means that the image of $e_\rho \overline\CH_c(G) e_\rho$ under $\tr_c$ lands in $ \overline R_c(G,\rho)$.

\begin{corollary} \label{c:rigid-type}
The restriction of the trace map $\tr_c: \overline{e_\rho \CH_c(G) e_\rho} \to \C F(\overline R_c(G,\rho))$ is surjective.
\end{corollary}
\begin{proof}
The category $\C C_{e_\rho}(G)$ is a direct product of finitely many Bernstein components \cite[(3.12)]{BK}. 
Without loss of generality, we may assume that $\C C_{e_\rho}(G)=\C C^{\fk s}(G)$ for a single Bernstein component $\fk s$. Then $\overline R_c(G,\rho)=\overline R_c(G,\fk s)$. 

Let $I_m$ be sufficiently small such that $e_\rho\CH(G) e_\rho\subset \CH(G,I_m)$. We use the methods from the proofs of  \cite[Lemmas 3.2, 3.3]{Ka1}. The category $\C C_{I_m}(G)$ splits into a finite direct product of subcategories:
\begin{equation}\C C_{I_m}(G)=\C C_{e_\rho}(G)\times \prod_{e'\neq e_\fk s} \C C_{e'}(G),
\end{equation}
for a finite set of Bernstein idempotents $\{e'\}$. This implies that 
\begin{equation}
K_0(G,I_m)=K_0(G,\rho)\oplus\bigoplus_{e'\neq e_\fk s} K_0(G,e'),
\end{equation}
in the obvious notation. Applying the Hattori-Stalling map $\tau$ and using Theorem \ref{t:HS}:
\begin{equation}
\overline H_c(G,I_m)=\overline{e_\rho \C H_c(G) e_\rho} \oplus\bigoplus_{e'} \overline\CH_c(G,e').
\end{equation}
By Corollary \ref{c:rigid}, $\tr_c$ maps $\overline\CH_c(G,e')$ onto $\C F(\overline R_c(G,\fk s'))$, where $\fk s'$ is the Bernstein component of $e'$. By Theorem \ref{t:rigid-PW}(1) 
\[\tr_c(\overline H_c(G,I_m))= \C F(\overline R_c(G,\fk s))\times\prod_{\fk s'\neq \fk s}\C F(\overline R_c(G,\fk s')),\]
and therefore $\tr_c(\overline{e_\rho \C H_c(G) e_\rho})=\C F(\overline R_c(G,\fk s)).$
\end{proof}

\smallskip

\begin{definition}\label{d:rigid-type}
We say that a type $(\C K,\rho)$ is \emph{rigid} if the composition of the natural maps
\[\bigoplus_{K\supseteq \C K} e_\rho H(K) e_\rho \rightarrow e_\rho \CH_c(G) e_\rho \twoheadrightarrow \overline{e_\rho \CH_c(G) e_\rho}
\]
is surjective.
\end{definition}

In section \ref{s:rigid-types}, we will discuss important examples of rigid types.

\section{The proofs of Theorem \ref{t:main} and Corollary \ref{c:main}}\label{s:proofs}
We are now in position to prove Theorem \ref{t:main}.

\subsection{} Firstly, we can sharpen Theorem \ref{t:signature} using the notion of rigid representations. Let $\CB_c(G,\fk s)=\{V_1,\dots,V_n\}$ denote a basis of $\overline R_c(G,\fk s)_\bQ$ consisting of irreducible tempered representations. From the previous section we know that $\overline R_c(G,\fk s)_\bQ$ is finite dimensional.

\begin{theorem}[The signature theorem II]\label{t:signature2} Let $(\pi,V)$ be an irreducible smooth $G$-representation with inertial support $\fk s$. Suppose that $V$ admits a $*$-invariant Hermitian form. Then there exist $w_1,\dots,w_n\in\bW\otimes_\bZ \bQ$ such that, for every compact open subgroup $K$, the $K$-signature of $V$ is
\[\sg_V^K=\sum_{i=1}^n w_i m_{V_i}^K.
\]
\end{theorem}
Notice that, in this version, we cannot say that the scalars $w_i$ are integral, only rational.
\begin{proof}
In light of Theorem \ref{t:signature}, it is sufficient to show that the $K$-character of every irreducible tempered representation $V'$ with inertial support $\fk s$ can be written as a rational combination (independent of $K$) of $\CB_c(G,\fk s).$ But this is equivalent to writing $V'$ in $\overline R_c(G,\fk s)_\bQ$ in terms of the basis $\CB_c(G,\fk s).$
\end{proof}

\begin{remark}
An important question is if there is a natural basis $\C B_c(G,\fk s)$ or more generally a basis of all of $\overline R_c(G)$ or $\overline R_c(G)_\bQ$. The analogy is that for $(\fg,K)$-modules of real reductive groups, a fundamental theorem of Vogan says that the set of irreducible tempered $(\fg,K)$-modules is a basis (over $\bZ$!) of the analogous quotient of the Grothendieck group of admissible $(\fg,K)$-modules. In the setting of $p$-adic groups, we may decompose (see for example \cite[section 4]{Da} or \cite[Proposition 6.5]{CH} for the Iwahori case)
\begin{equation}
\overline R_c(G)_\bQ=\left(\bigoplus_{M\le G } i_M^G(\overline R_\el(M)_{\bQ,\C X(M)})\right)/\sim,
\end{equation}
where $\overline R_\el(M)_{\bQ,\C X(M)}$ denotes the space of $\C X(M)$-coinvariants of $\overline R_\el(M)_\bQ$, the direct sum is over the set of standard Levi subgroups $M$, and $(M,\sigma)\sim (M',\sigma')$ if there exists $w\in W(G,A)$ such that $M'=w(M)$ and $\sigma'=w(\sigma)$. However, beyond this point, we do not know how to make canonical choices for the basis of $\overline R_\el(M)_{\bQ,\C X(M)}$. In fact, the example of the two (Iwahori-spherical) tempered direct summands of the minimal principal series representation of $SL(2,\bQ_p)$ induced from the unramified quadratic character appears to suggest that a canonical choice for irreducible elliptic representations may not be possible.
\end{remark}

\subsection{} Let $e$ be an idempotent in $\CH(G)$ and $(\pi,V)$ an irreducible representation in $\C C_e(G)$. Suppose that $(\pi,V)$ carries a hermitian $*$-invariant form $\<~,~\>_V$. If $e^*=e$, we may restrict $\<~,~\>_V$ to $\pi(e)V$ and get an invariant hermitian form $\<~,~\>_{V,e}$:
\[
\<\pi(e f e)\pi(e) v_1,\pi(e)v_2\>_{V,e}=\<\pi(e) v_1,\pi(e^* f^* e^*)\pi(e)v_2\>_{V,e}=\<\pi(e) v_1,\pi(e f^* e)\pi(e)v_2\>_{V,e}.
\]
Since $e$ is a self-adjoint idempotent
\[\<\pi(e)v_1,v_2\>_V=\<\pi(e)\pi(e)v_1,v_2\>_V=\<\pi(e)v_1,\pi(e^*)v_2\>_V=\<\pi(e)v_1,\pi(e)v_2\>_v.
\]
This implies that if $\<~,~\>_V$ is nondegenerate, so is $\<~,~\>_{V,e}$. Of course, if $\<~,~\>_V$ is positive definite, so is $\<~,~\>_{V,e}$, which means that
\begin{equation}
\text{if } V \text{ is unitary then }\pi(e)V\text{ is unitary}.
\end{equation}
Conversely, suppose that $(\bar \pi,U)$ is an $e\CH e$-module with an $e\CH e$-invariant hermitian form $\<~,~\>_U$. The functor from $e\CH e$-modules to $\CH$-modules takes $U$ to $\CH e\otimes_{e\CH e}U$. For $f_1,f_2\in \CH e$, define
\begin{equation}
\<f_1\otimes u_1,f_2\otimes u_2\>=\<\bar \pi(e f_2^* f_1 e) u_1,u_2\>_U,
\end{equation}
an invariant hermitian form on $\CH e\otimes_{e\CH e}U$.

\subsection{}The difficult part of the unitarity equivalence is to show that if  $(\pi,V)$ carries a nondegenerate hermitian $*$-invariant form $\<~,~\>_V$ such that $\<~,~\>_{V,e}$ is positive definite, then $\<~,~\>_V$ is positive definite, i.e., if $\pi(e)V$ is unitary, then $V$ is unitary.

For this, we need to restrict to the case of types. Let $(\C K,\rho)$ be a type as before and let $(\pi,V)$ be an irreducible representation in $\C C_{e_\rho}(G)$.

For every compact open subgrup $K\supseteq \C K$, we have defined the notions of $K$-multiplicity and $K$-signature of $V,\<~,~\>_V$. Define the notions of $e_\rho H(K)e_\rho$-multiplicty and signature of $\pi(e_\rho)V,\<~,~\>_{V,e_\rho}$:
\begin{equation}
m_{\pi(e_\rho)V}^{e_\rho H(K)e_\rho},\quad \sg_{\pi(e_\rho)V}^{e_\rho H(K)e_\rho},
\end{equation}
in exactly the same way. Using Theorem \ref{t:signature2}, write
\begin{equation}
\sg_V^K=\sum_{i=1}^n w_i m_{V_i}^K=\sum_{i=1}^n a_i m_{V_i}^K+s\sum_{i=1}^n b_i m_{V_i}^K,
\end{equation}
for certain $a_i,b_i\in \bZ$, independent of $K$. The signature $\sg_{\pi(e_\rho)V}^{e_\rho H(K)e_\rho}$ is obtained by simply restricting to the $K$-types $(\mu,E_\mu)$ such that $\pi(e_\rho)E_\mu\neq 0$. Hence
\begin{equation}
\sg_{\pi(e_\rho)V}^{e_\rho H(K)e_\rho}=\sum_{i=1}^n a_i m_{\pi(e_\rho)V_i}^{e_\rho H(K)e_\rho}+s\sum_{i=1}^n b_i m_{\pi(e_\rho)V_i}^{e_\rho H(K)e_\rho}
\end{equation}
If $\<~,~\>_{V,e_\rho}$ is positive definite, it implies that
\begin{equation}\label{e:zero}
\sum_{i=1}^n b_i m_{\pi(e_\rho)V_i}^{e_\rho H(K)e_\rho}=0, \text{ equivalently }m_{\pi(e_\rho)\C V}^{e_\rho H(K)e_\rho}=0,\text{ where }\C V=\sum_{i=1}^n b_i V_i\in \overline R_c(G,\rho).
\end{equation}
The condition in (\ref{e:zero}) holds for all $K\supset\C K$, and therefore, when $(\C K,\rho)$ is rigid, it can be rephrased in terms of the notation in section \ref{s:rigid} as saying that
\begin{equation}
\tr(f,\C V)=0\text{ for all } f\in \overline{e_\rho \CH_c e_\rho}.
\end{equation}
Then Corollary \ref{c:rigid-type} implies that
\begin{equation}
\C V=0 \text{ in } \overline R_c(G,\rho).
\end{equation}
Now since $\{V_i\}$ is a basis of $\overline R_c(G,\rho)$, we get that
\begin{equation}
b_i=0\text{ for all }i,
\end{equation}
and so $\<~,~\>_V$ is also positive definite.\qed

\subsection{}Remark that the method of proof fails if one only considers the signature characters with respect to a fixed maximal compact open subgroup, for example the maximal special subgroup $K_0$. The reason is that 
\begin{equation}
\sum_{i=1}^n b_i m_{\pi(e_\rho)V_i}^{e_\rho H(K_0)e_\rho}=0
\end{equation}
does not necessarily imply that $\sum_{i=1}^n b_i V_i=0$  in  $\overline R_c(G,\rho).$ For example, take $G=SL(2,\bQ_p)$, $K_0=SL(2,\bZ_p)$, and $(\C K,\rho)=(I,1_I)$, where $I\subset K_0$ is an Iwahori subgroup. In that case, $\overline R_c(G,e_I)$ is $3$-dimensional, but there are only two irreducible $K_0$-types with $I$-fixed vectors.

\subsection{}Corollary \ref{c:main} follows from Theorem \ref{t:main} via a well-known argument, as applied for example in \cite[section 4]{BHK} to deduce the preservation of Plancherel measures. We include it here for completeness. We follow the notation in \cite{BHK}. Recall that a normalized Hilbert algebra $A$ is an associative unital $\bC$-algebra with a star operation $*$ and an inner product $[~,]_A$ such that $[1_A,1_A]=1$. We do not reproduce the axioms of compatibility between $[~,~]_A$ and $*$, but we refer to \cite[Definition 3.1]{BHK} for the details.

There are three normalized Hilbert algebras that enter in the picture for a type $(\C K,\rho)$ such that $e_\rho^*=e_\rho$. Recall that $(\rho,W)$ is an irreducible smooth $\C K$-representation with contragredient $(\rho^\vee,W^\vee)$. The first algebra is $e_\rho\CH e_\rho$  where the Hilbert product is
\begin{equation}
[a,b]_{e_\rho\CH e_\rho}=\frac 1{\dim W} (a^*\star b)(1_G),\quad a,b\in e_\rho\CH e_\rho.
\end{equation}
The second is 
\[\C E_\rho=\End_\bC[W].
\]
We fix a positive definite $\C K$-invariant form $\<~,~\>_{\rho}$ on $W$ and define a star operation $a\mapsto a^*$ on $\C E_\rho$ via
\begin{equation}
\<a^*(v),w\>_{\rho}=\<v,a(w)\>_\rho,\quad \text{for all }v,w\in W.
\end{equation}
The inner product is $[~,~]_{E_\rho}$:
\begin{equation}
[a,b]_{E_\rho}=\frac 1{\dim W}\tr_{W}(a^*b).
\end{equation}
The same definitions apply to $\C E_{\rho^\vee}=\End_\bC[W^\vee].$ The transpose map $a\to a^t$ defines an isomorphism of Hilbert algebras between $\C E_\rho$ and $\C E_{\rho^\vee}$.

Finally, $\CH(G,\rho)$ is also a Hilbert algebra with involution $h\mapsto h^*$ defined by
\begin{equation}
h^*(x)=(h(x^{-1}))^*, \quad x\in G,
\end{equation}
and inner product
\begin{equation}
[h_1,h_2]_{\CH(G,\rho)}=\frac{\mu(\C K)}{\dim W}\tr_{W^\vee}((h_1^*\star h_2)(1_G)).
\end{equation}
As in \cite[\S4.4]{BHK}, define the tensor product 
\begin{equation}
\CH_W(G,\rho)=\CH(G,\rho)\otimes_{\bC} \C E_\rho,
\end{equation}
and endow it with a Hilbert algebra structure by using the product star operation and the product inner product. For every pair $(h,a)\in \CH(G,\rho)\times_{\bC} \C E_\rho$, define the function 
\begin{equation}
f_{(h,a)}(x)=\dim W \tr_{W^\vee}(h(x)a^t),\quad x\in G.
\end{equation}
\begin{proposition}[{\cite[Proposition 4.4]{BHK}}]\label{p:Hilbert}
The assignment $(h,a)\mapsto f_{(h,a)}$ induces an isomorphism of normalized Hilbert algebras 
\[\C H_W(G,\rho)=\CH(G,\rho)\otimes \C E_\rho\cong e_\rho\CH e_\rho.
\]
\end{proposition}
In particular, we have a natural bijection between the unitary modules for the two algebras. Finally, the correspondence
\begin{equation}\label{e:Morita}
M\mapsto M\otimes W
\end{equation}
is an equivalence of categories between $\CH(G,\rho)$-modules and $\CH_W(G,\rho)$-modules. Since $W$ is a unitary $\C E_\rho$-module, this equivalence restricts to a natural bijection of unitary modules. Thus Proposition \ref{p:Hilbert} together with (\ref{e:Morita}) imply that the natural equivalence of categories between $e_\rho\CH e_\rho$-modules and $\C H(G,\rho)$-modules induces a bijection of unitary modules. Hence Corollary \ref{c:main} follows from Theorem \ref{t:main}.\qed

\section{Rigid types}\label{s:rigid-types}

We retain the notation from section \ref{s:rigid}. 

\subsection{} Let $(\C K,\rho)$ be a type and suppose that there exists an Iwahori subgroup $I$ such that $\C K\subseteq I$. Let $\C P_{\max}$ be the set of maximal compact open subgroups (maximal parahoric subgroups) containing $I$. Since $\C P_{\max}$ provide a complete set of representatives for the $G$-conjugacy classes of maximal compact open subgroups of $G$, the natural map
\[
\bigoplus_{K\in \C P_{\max}} \C H(K)\to \overline \CH_c(G) 
\]
is surjective and therefore
\begin{equation}
\bigoplus_{K\in \C P_{\max}} e_\rho\C H(K)e_\rho\to \overline {e_\rho\CH_c(G) e_\rho}
\end{equation}
is surjective. Since $\C K\subset K$ for every $K\in \C P_{\max}$, this shows that $(\C K,\rho)$ is rigid in the sense of Definition \ref{d:rigid-type}.

\subsection{} We now look at the case of level zero types \cite{Mo}, \cite{Lu2}. We refer to the {\it loc. cit.} for the necessary structural results from Bruhat-Tits theory, in particular, about parahoric subgroups. Fix an Iwahori subgroup $I$, and recall the Bruhat decomposition
\[G=\sqcup_{w\in W} I \dot w I,
\]
where $W$ is the Iwahori-Weyl group and $\dot w$ is a representative in $G$ of $w$. Let $\Pi$ be the set of affine simple roots defined by $I$ and $W'$ the affine Weyl group. We have $W=W'\rtimes\Omega$, where $\Omega$ can be identified with the subgroup of $W$ of elements of length zero. Let $P\supseteq I$ be a parahoric subgroup. It corresponds to a subset $J\subsetneq\Pi$ and we write $P=P_J$ to emphasize this relation. Let $W_J$ denote the (finite) subgroup of $W$ generated by the reflections $s_\al$, $\al\in J$. There is a one-to-one correspondence 
\begin{equation}
W_J\backslash W/W_J\longleftrightarrow P_J\backslash G/P_J,\ W_J w W_J\mapsto P_J\dot wP_J=\cup_{u\in W_J w W_J} I\dot u I.
\end{equation}
Define
\begin{equation}
W(J)=\{w\in N_W(W_J)\mid w\text{ is the minimal length representative of  }W_JwW_J\}.
\end{equation}
The set $W(J)$ is a subgroup of $W$. 

\begin{definition}
A level zero type is a pair $(P,\rho)$, where $P=P_J$ is a parahoric subgroup and $\rho$ is a representation of $P_J$ inflated from a cuspidal representation of the finite reductive quotient $M_J$ of $P_J$.
\end{definition}

The main result of \cite{Mo} is a description of the $\rho$-spherical Hecke algebra $\CH(G,\rho)$. The starting essential observation \cite[Theorem 4.15]{Mo} is that a coset $P\dot w P$, with $w$ of minimal length in $W_J w W_J$, supports a nonzero element of $\CH(G,\rho)$ only if
\begin{equation}
w\in N_W(W_J)\text{ and } ^w\rho\cong\rho.
\end{equation}
Set 
\begin{equation}
W(\rho)=\{w\in W(J)\mid ~^w\rho\cong\rho\}.
\end{equation}
The element $\dot w$ is compact if and only if $w$ has finite order in $W$. Therefore  the compact cocenter of $\CH(G,\rho)$ (respectively, of $e_\rho\CH(G) e_\rho$) is generated by elements (respectively, functions) supported on cosets
\begin{equation}
P_J\dot w P_J,\ w\in W(\rho)\text{ of finite order}.
\end{equation}
Suppose $P_J\dot w P_J$ is such a coset. Since $w\in N_W(W_J)$ and $w$ has finite order, the subgroup $\langle w, W_J\rangle$ of $W$ is finite. By the Bruhat-Tits fixed point theorem \cite[(3.2.4)]{BT}, there exists a subset $J'\subsetneq\Pi$ with $J\subseteq J'$ and $\Omega'$ a finite subgroup of $\Omega$ such that $\langle w, W_J\rangle\subset W_{J'}\rtimes\Omega'$. Set 
\begin{equation}
K_{J'}=\cup_{u\in W_{J'}\rtimes\Omega'} I\dot u I.
\end{equation}
This is a compact open subgroup of $G$ and $P_J\subset K_{J'},\ P_J\dot w P_J\subset K_{J'}.$ This shows:

\begin{lemma}
Every level zero type $(P,\rho)$ is rigid, hence Theorem \ref{t:main} applies to $(P,\rho)$.
\end{lemma}

As shown in \cite[\S 7.3]{Mo}, the group $W(\rho)$ can be decomposed canonically as
\begin{equation}
W(\rho)=R(\rho)\rtimes C(\rho),
\end{equation}
where $R(\rho)$ is $1$ if $|\Pi\setminus J|=1$, and when $|\Pi\setminus J|\ge 2$, $R(\rho)$ is an affine Weyl group generated by a certain finite set of simple reflections $S(\rho)$. Let $\Delta$ be the set of simple affine roots corresponding to $S(\rho)$ and denote by $v_a\in S(\rho)$ the reflection corresponding to $a\in\Delta$. The only property that we need to record about $S(\rho)$ is that each element of $S(\rho)$ lies in a $W_K$ where $J\subsetneq K$. Let $\mu: W(\rho)\times W(\rho)\to \bC^\times$ be the 2-cocycle computed in \cite[\S7.11]{Mo}. It has the following properties \cite[Lemma 6.2, Lemma 7.11]{Mo}:
\begin{equation}
\begin{aligned}
&\mu(sv,sw)=\mu(s,t),\ s,t\in C(\rho), &v,w\in R(\rho);\\
&\mu(w,1)=\mu(1,w)=1,\ \mu(w,w^{-1})=1, &w\in W(\rho);\\
&\mu(w_1,w_2)\cdot \mu(w_2^{-1},w_1^{-1})=1, &w_1,w_2\in W(\rho).
\end{aligned}
\end{equation}

\begin{theorem}\label{t:morris} The algebra $\CH(G,\rho)$ has a basis given by elements $T_w$, $w\in W(\rho)$, where $T_w$ is supported on the coset $P\dot w P$, subject to the following relations. Let $w\in W(\rho)$, $t\in C(\rho)$, $v_a\in S(\rho)$, $a\in \Delta$. Then:
\begin{enumerate}
\item $T_w T_t=\mu(w,t) T_{wt}$; $T_t T_w=\mu(t,w) T_{tw}$.
\item $T_{v_a}T_w=\begin{cases}T_{v_a w},&\text{if } w^{-1}(a)>0,\\ p_a T_{v_a w}+(p_a-1)T_w,&\text{if }w^{-1}(a)<0.\end{cases}$
\item $T_{w}T_{v_a}=\begin{cases}T_{wv_a},&\text{if } w(a)>0,\\ p_a T_{wv_a}+(p_a-1)T_w,&\text{if }w(a)<0.\end{cases}$
\end{enumerate}
Here $p_a\neq 1$ are certain nonnegative powers of the residual characteristic.
\end{theorem}

In other words, $\CH(G,\rho)$ is the smash-product of an affine Hecke algebra with a twisted group algebra. This means that classification of unitary representations of level zero of the $p$-adic group $G$ is equivalent with the classification of the unitary dual of the Hecke algebras in Theorem \ref{t:morris}. When $G$ is adjoint and $\rho$ is unipotent, Lusztig \cite{Lu2} gives a complete description of $\CH(G,\rho)$. In that case, the cocycle $\mu$ is trivial, and $\CH(G,\rho)$ is an affine Hecke algebra with unequal parameters. 

\begin{remark}
When the cocycle $\mu$ is trivial, a precise description of the cocenter and compact cocenter of $\CH(G,\rho)$ is available by \cite{CH,HN}.
\end{remark}

\begin{remark}The dichotomy (positive versus zero levels) in the previous discussion is motivated by the unrefined minimal K-type theory of Moy and Prasad \cite[Theorem 5.2]{MP}.
\end{remark}

\subsection{}Suppose now that $G=GL(n,F)$ and $(\C K,\rho)$ is a type of $G$. In this case, there exists a unique conjugacy class of maximal compact subgroups represented by $K_0=GL(n,\C O)$, where $\CO$ is the ring of integers in $F$. Without loss of generality, we may assume that $\C K\subseteq K_0.$ Since every element in $\overline \CH_c(G)$ is represented by a function supported on $K_0$,  $\C K$ is trivially rigid.

\ifx\undefined\bysame
\newcommand{\bysame}{\leavevmode\hbox to3em{\hrulefill}\,}
\fi

\end{document}